\documentclass[11pt,reqno]{amsart}
\usepackage[T1]{fontenc}
\usepackage[utf8]{inputenc}
\usepackage{lmodern}
\usepackage[margin=1.05in]{geometry}
\usepackage{amsmath,amssymb,amsthm,mathtools}
\usepackage{booktabs,array}
\usepackage{microtype}
\usepackage[hidelinks]{hyperref}
\hypersetup{pdftitle={Ehrhart reciprocity and forced factors in three plane-partition enumerators},
pdfauthor={Yinuo Cheng},
pdfkeywords={plane partitions, Ehrhart reciprocity, irreducibility}}
\numberwithin{equation}{section}
\newtheorem{theorem}{Theorem}[section]
\newtheorem{proposition}[theorem]{Proposition}
\newtheorem{lemma}[theorem]{Lemma}

\theoremstyle{definition}
\newtheorem{definition}[theorem]{Definition}

\theoremstyle{remark}
\newtheorem{remark}[theorem]{Remark}
\newcommand{\Z}{\mathbb Z}
\newcommand{\Q}{\mathbb Q}
\newcommand{\R}{\mathbb R}
\newcommand{\NN}{\mathbb Z_{\geq0}}
\newcommand{\one}{\mathbf 1}
\newcommand{\relint}{\operatorname{relint}}

\newcommand{\aff}{\operatorname{aff}}
\newcommand{\den}{\operatorname{den}}

\newcommand{\qsp}{p_a}
\newcommand{\qtp}{q_a}
\newcommand{\stp}{s_a}
\newcommand{\QS}{\mathcal A}
\newcommand{\QT}{\mathcal Q}
\newcommand{\ST}{\mathcal S}
\newcommand{\eps}{\varepsilon}

\title[Ehrhart reciprocity for three plane-partition enumerators]
{Ehrhart reciprocity and forced factors in three plane-partition enumerators}
\author{Yinuo Cheng}
\address{School of Mathematical Sciences,
Capital Normal University, Beijing 100048, P.R. China}
\email{2250502145@cnu.edu.cn}

\date{}
\subjclass[2020]{Primary 05A15; Secondary 05A17, 52B20, 06A07}
\keywords{Plane partition, quasi-symmetry, Ehrhart polynomial,
Ehrhart--Macdonald reciprocity, lattice polytope, coefficient denominator}

\begin{document}
\begin{abstract}
We study three plane-partition enumerators arising from Schreier-Aigner's
quasi-symmetry classes. Their realizations as lattice-point enumerators,
together with staircase translations of interior lattice points, yield
factorizations by Ehrhart--Macdonald reciprocity. We determine the
consecutive linear factors, the parity and degree of the residual
polynomials, and explicit divisibility bounds for their coefficient
denominators. The denominator argument includes the half-integral
translation required by the second-kind classes. We also derive a corrected
size-five formula for the symmetric second-kind class. Exact computations
verify irreducibility of the quasi-symmetric residual
polynomials for every size from $3$ to $24$.
\end{abstract}
\maketitle

\section{Introduction}\label{sec:intro}

Schreier-Aigner introduced several quasi-symmetry classes in the study of
fully complementary higher dimensional partitions
\cite[Section~4]{SchreierAigner}. The enumerations in that paper suggest
factorizations for three ordinary plane-partition enumerators. These expressions
contain a string of consecutive linear factors, a possible additional factor
at the center of that string, and an even residual polynomial. They also
contain irreducibility assertions and a qualitative assertion about small
prime factors in coefficient denominators.

Theorems~\ref{thm:qs}--\ref{thm:st} establish the factorizations and parity
statements, determine the degrees, and give explicit denominator bounds.
Each enumerator is realized as the Ehrhart polynomial of a lattice
polytope whose interior lattice points are obtained by staircase
translation. The quasi-symmetric polytope is an order polytope;
the second-kind models have signed coordinate identifications.
Section~\ref{sec:limits} reports exact finite-range irreducibility
computations. Irreducibility for arbitrary admissible sizes remains open.

We use $h$ for the \emph{half-height} of a second-kind plane partition:
both the box height and the complementary sum are $2h$. The definition
in \cite[Section~4.1]{SchreierAigner} specifies a box height $2c$
but a complementary sum $c$. Our convention agrees with the enumerations
in its Appendices A.2 and A.3; see Remark~\ref{rem:normalization}.

\subsection{The three families}
Write $[a]=\{1,\ldots,a\}$, with $a\geq1$. An array in an $(a,a,H)$-box is a
matrix $\pi=(\pi_{ij})_{(i,j)\in[a]^2}$ of integers. It is a plane partition
if
\begin{equation}\label{eq:pp}
\begin{aligned}
&0\leq\pi_{ij}\leq H &&(1\leq i,j\leq a),\\
&\pi_{ij}\geq\pi_{i+1,j} &&(1\leq i<a,\ 1\leq j\leq a),\\
&\pi_{ij}\geq\pi_{i,j+1} &&(1\leq i\leq a,\ 1\leq j<a).
\end{aligned}
\end{equation}
The height is an upper bound; it need not be attained. Define the maps
\begin{equation}\label{eq:involutions}
\begin{split}
s(i,j)&=(j,i),\\
\rho(i,j)&=(a+1-i,a+1-j),\\
\tau(i,j)&=(a+1-j,a+1-i)=\rho s(i,j).
\end{split}
\end{equation}
Each is an involution. The maps $s$ and $\rho$ commute. The fixed cells of
$\tau$ lie on the anti-diagonal $i+j=a+1$.

\begin{definition}\label{def:families}
For $H,h\in\NN$, define the following sets and their cardinalities.
\begin{enumerate}
\item $\QS_a(H)$ consists of the plane partitions in an $(a,a,H)$-box
such that
\begin{equation}\label{eq:qs}
\pi_{ij}=\pi_{ji}\qquad\text{if }i+j\ne a+1.
\end{equation}
Set $A_a(H)=|\QS_a(H)|$.
\item $\QT_a(h)$ consists of the plane partitions in an $(a,a,2h)$-box
such that
\begin{equation}\label{eq:qt}
\pi_{ij}+\pi_{\tau(i,j)}=2h\qquad\text{if }i\ne j.
\end{equation}
Set $Q_a(h)=|\QT_a(h)|$.
\item $\ST_a(h)$ consists of the elements of $\QT_a(h)$ satisfying
\begin{equation}\label{eq:sym}
\pi_{ij}=\pi_{ji}\qquad\text{for all }i,j.
\end{equation}
Set $S_a(h)=|\ST_a(h)|$.
\end{enumerate}
\end{definition}

The exception in \eqref{eq:qs} is the anti-diagonal, whereas the exception
in \eqref{eq:qt} is the main diagonal. In particular,
\eqref{eq:qt} imposes no relation between two main-diagonal entries, even
when their cells are exchanged by $\tau$. All three sets contain exactly
the zero matrix at height zero. No quarter-complementarity or prescribed
volume condition is imposed.

\begin{remark}\label{rem:normalization}
Take $a=2$ and $h=1$. The two off-diagonal cells are fixed by $\tau$.
Equation~\eqref{eq:qt} therefore forces both entries to be $1$. The arrays
are precisely
\[
\begin{pmatrix}u&1\\1&t\end{pmatrix},
\qquad u\in\{1,2\},\quad t\in\{0,1\}.
\]
Consequently $Q_2(1)=S_2(1)=4$, agreeing with the corresponding entries in
Appendices A.2 and A.3 of \cite{SchreierAigner}. If the complementary sum
were $h=1$ instead, either fixed off-diagonal cell would have to satisfy
$2\pi_{ij}=1$, so the count would be zero.
\end{remark}

\subsection{Main results}
For a rational polynomial, $\den(f)$ denotes the least positive integer
$D$ such that $Df$ has integer coefficients. Equivalently, it is the least
common multiple of the reduced denominators of its coefficients. Put
\begin{equation}\label{eq:dimensions}
N_a=\frac{a(a+1)}2+\left\lfloor\frac a2\right\rfloor,
\quad d_a=\left\lceil\frac{a^2}{2}\right\rceil,
\quad e_a=\left\lfloor\frac{(a+1)^2}{4}\right\rfloor.
\end{equation}
We use the polynomial convention
\[
\binom{x}{k}=\frac{x(x-1)\cdots(x-k+1)}{k!},\qquad
\binom{x}{0}=1.
\]
Thus the variable $c$ in the next theorems is a polynomial variable. A
counting interpretation is asserted only when the corresponding height is
a nonnegative integer.

\begin{theorem}[Quasi-symmetric class]\label{thm:qs}
For each $a\geq1$, $A_a(H)$ extends uniquely to a rational polynomial of
degree $N_a$. Let $\alpha_a=1$ for even $a$ and $\alpha_a=0$ for odd $a$.
There is a unique nonzero even polynomial $\qsp(c)\in\Q[c]$ such that
\begin{equation}\label{eq:main-qs}
A_a(c-a)=c^{\alpha_a}\binom{c+a-1}{2a-1}\qsp(c).
\end{equation}
Its degree and denominator satisfy
\begin{equation}\label{eq:main-qs-degree-den}
\deg\qsp=\left\lfloor\frac{(a-1)^2}{2}\right\rfloor,
\qquad
\den(\qsp)\ \bigm|\ \frac{N_a!}{(2a-1)!}.
\end{equation}
\end{theorem}

\begin{theorem}[Second-kind class]\label{thm:qt}
For each $a\geq2$, $Q_a(h)$ extends uniquely to a rational polynomial of
degree $d_a$. There is a unique nonzero even polynomial
$\qtp(c)\in\Q[c]$ such that
\begin{equation}\label{eq:main-qt}
Q_a(c-a/2)=c\binom{c+a/2-1}{a-1}\qtp(c).
\end{equation}
Writing $\ell_a=d_a-a$, we have
\begin{equation}\label{eq:main-qt-den}
\deg\qtp=\ell_a,
\qquad
\den(\qtp)\ \bigm|\
\begin{cases}
d_a!/(a-1)!,&a\text{ even},\\
2^{\ell_a}d_a!/(a-1)!,&a\text{ odd}.
\end{cases}
\end{equation}
\end{theorem}

\begin{theorem}[Symmetric second-kind class]\label{thm:st}
For each $a\geq2$, $S_a(h)$ extends uniquely to a rational polynomial of
degree $e_a$. Set $\beta_a=0$ if $a\equiv3\pmod4$ and $\beta_a=1$
otherwise. There is a unique nonzero even polynomial
$\stp(c)\in\Q[c]$ such that
\begin{equation}\label{eq:main-st}
S_a(c-a/2)=c^{\beta_a}\binom{c+a/2-1}{a-1}\stp(c).
\end{equation}
With $m_a=e_a-a+1-\beta_a$, its degree and denominator satisfy
\begin{equation}\label{eq:main-st-den}
\deg\stp=m_a,
\qquad
\den(\stp)\ \bigm|\
\begin{cases}
e_a!/(a-1)!,&a\text{ even},\\
2^{m_a}e_a!/(a-1)!,&a\text{ odd}.
\end{cases}
\end{equation}
\end{theorem}

Every denominator prime in these theorems is at most the corresponding
dimension $N_a$, $d_a$, or $e_a$. The extra powers of $2$ in
\eqref{eq:main-qt-den} and \eqref{eq:main-st-den} do not change this
conclusion, since the last two dimensions are at least $2$. The bounds
need not be sharp; they give an explicit restriction on the prime divisors
of all coefficient denominators.

\subsection{Notation}
The three counting functions and their residual polynomials are paired as
follows:
\[
\begin{array}{c|ccc}
\text{family}&\text{counting function}&\text{residual}&\text{dimension}\\
\hline
\text{quasi-symmetric}&A_a&p_a&N_a\\
\text{second-kind}&Q_a&q_a&d_a\\
\text{symmetric second-kind}&S_a&s_a&e_a
\end{array}
\]
We write $\one$ for the all-ones matrix. The two staircases are
\[
\delta_{ij}=2a+1-i-j,\qquad
\eta_{ij}=a+1-i-j,\qquad \delta=\eta+a\one.
\]
For a partition $\pi$ in a box of height $H$, its complement is denoted
by $\bar\pi$:
\[
\bar\pi_{ij}=H-\pi_{a+1-i,a+1-j}.
\]
The notation $\pi'=\pi-\delta$ is used when subtracting the staircase.
For a second-kind partition of half-height $h$, we put
\[
z=\pi-h\one.
\]
If $\pi'$ has half-height $h-a$, we similarly put
$z'=\pi'-(h-a)\one$. Thus $\pi,\pi'$ are uncentered partitions and
$z,z'$ are their centered matrices, respectively.

In the polytope models, $u$ and $z$ denote real matrices.
For a fixed matrix, the associated threshold matrices are denoted
$E_t(u)$ and $T_t(z)$, with entries $E_t(u)_{ij}$ and $T_t(z)_{ij}$.

\section{Lattices and Ehrhart theory}
\label{sec:prelim}

\begin{definition}
Let $V$ be a real vector space of dimension $d$. A full lattice
$\Lambda\subset V$ is the set of all integer linear combinations of some
real basis of $V$. A convex combination of $v_1,\ldots,v_k$ is a sum
$\sum_{i=1}^k\lambda_i v_i$ with $\lambda_i\geq0$ and
$\sum_{i=1}^k\lambda_i=1$; their convex hull is the set of all such
combinations. A polytope is the convex hull of finitely many points.
A vertex, or extreme point, of a convex set $P$ is a point $x\in P$
such that $x=\lambda v+(1-\lambda)w$, with $v,w\in P$ and
$0<\lambda<1$, implies $v=w=x$.
A polytope is a lattice polytope with respect to $\Lambda$ if every vertex belongs
to $\Lambda$. Its dimension is the dimension of its affine hull.
For $t\geq0$, its dilation is $tP=\{tx:x\in P\}$. The relative interior
$\relint(P)$ is the interior taken in $\aff(P)$, not in a larger ambient
space.
\end{definition}

Choosing a lattice basis identifies $(V,\Lambda)$ with $(\R^d,\Z^d)$.
For the matrix spaces used below, we obtain such coordinates by selecting
independent entries.

\begin{lemma}[Strict inequalities and relative interior]\label{lem:strict}
Let $V$ be a finite-dimensional real vector space and let
\[
P=\{x\in V:f_j(x)\geq0\text{ for }1\leq j\leq M\},
\]
where each $f_j$ is affine. Suppose that some $x^*\in P$ satisfies
$f_j(x^*)>0$ for every $j$. Then $\aff(P)=V$, and
\[
\relint(P)=\{x\in V:f_j(x)>0\text{ for every }j\}.
\]
\end{lemma}
\begin{proof}
There are finitely many inequalities, and each $f_j$ is continuous.
Consequently a sufficiently small open ball in $V$ around $x^*$ is
contained in $P$, which proves $\aff(P)=V$. The same argument shows that
every point satisfying all inequalities strictly is interior.
Conversely, suppose $f_j(x)=0$ for some $x\in P$. Write
$f_j(y)=b_j+\lambda_j(y)$, with $\lambda_j$ linear. The strict inequality
at $x^*$ implies
$\lambda_j(x^*-x)=f_j(x^*)-f_j(x)>0$. For every $\epsilon>0$,
\[
f_j\bigl(x-\epsilon(x^*-x)\bigr)=-\epsilon f_j(x^*)<0.
\]
Thus every neighborhood of $x$ in $V$ contains a point outside $P$, so
$x$ is not interior. This proves both inclusions.
\end{proof}

We recall Ehrhart's polynomiality theorem \cite{Ehrhart} and
Ehrhart--Macdonald reciprocity \cite{Macdonald}, in the formulations
of Beck and Robins \cite[Theorems 3.8 and 4.1]{BeckRobins}.

\begin{theorem}[Ehrhart {\cite[Theorem~3.8]{BeckRobins}}]\label{thm:ehrhart}
Let $P$ be a full-dimensional lattice polytope in a rank-$d$ lattice
$\Lambda$. For integers $t\geq0$, set $L_P(t)=\#(tP\cap\Lambda)$.
There is a unique polynomial $L_P\in\Q[t]$ of degree exactly $d$ with
these values, and $L_P(0)=1$.
\end{theorem}

\begin{theorem}[Ehrhart--Macdonald {\cite[Theorem~4.1]{BeckRobins}}]\label{thm:reciprocity}
Under the hypotheses of Theorem~\ref{thm:ehrhart}, for each positive
integer $t$ one has
\begin{equation}\label{eq:EM}
L_P(-t)=(-1)^d L_P^\circ(t),\qquad
L_P^\circ(t)=\#(\relint(tP)\cap\Lambda).
\end{equation}
The left side is evaluation of the polynomial at a negative integer.
\end{theorem}

\section{Factorization and denominator bounds}
\label{sec:general}

\begin{theorem}\label{thm:general}
Let $P\subset V$ be a full-dimensional lattice polytope in a rank-$d$
lattice $\Lambda$. Suppose that $r\geq1$ is an integer and
$\omega\in\Lambda$ satisfies
\begin{equation}\label{eq:translation-general}
\relint(tP)\cap\Lambda
=\omega+\bigl((t-r)P\cap\Lambda\bigr)\qquad(t\in\Z,\ t\geq r),
\end{equation}
and that $\relint(tP)\cap\Lambda$ is empty for $1\leq t<r$.
Define
\begin{equation}\label{eq:general-FB}
F(c)=L_P(c-r/2),\qquad
B_r(c)=\binom{c+r/2-1}{r-1}.
\end{equation}
Let $\eps\in\{0,1\}$ be congruent to $d-r+1$ modulo $2$. Then a unique
nonzero even polynomial $p\in\Q[c]$ satisfies
\begin{equation}\label{eq:general-factor}
F(c)=c^\eps B_r(c)p(c).
\end{equation}
Its degree is $\ell=d-r+1-\eps\geq0$, and
\begin{equation}\label{eq:general-den}
\den(p)\ \bigm|\
\begin{cases}
d!/(r-1)!,&r\text{ even},\\
2^\ell d!/(r-1)!,&r\text{ odd}.
\end{cases}
\end{equation}
In particular, every denominator prime is at most $\max\{2,d\}$.
\end{theorem}

\subsection{The root string and the parity}
We first prove \eqref{eq:general-factor}. By reciprocity and
\eqref{eq:translation-general}, for every integer $t\geq r$,
\[
L_P(-t)=(-1)^dL_P(t-r).
\]
Both sides are polynomials in $t$. Their difference has infinitely many
roots, and therefore is the zero polynomial. Thus
\begin{equation}\label{eq:general-reflection}
L_P(-t)=(-1)^dL_P(t-r)\qquad\text{in }\Q[t].
\end{equation}
The empty-interior hypothesis and \eqref{eq:EM} also give
\begin{equation}\label{eq:negative-roots}
L_P(-j)=0\qquad(1\leq j\leq r-1).
\end{equation}
The values of $c$ corresponding to these roots are $c=r/2-j$. They are
distinct, and
\begin{equation}\label{eq:B-product}
B_r(c)=\frac{1}{(r-1)!}\prod_{j=1}^{r-1}(c-r/2+j).
\end{equation}
Consequently the factor theorem, applied successively at these roots,
implies $B_r\mid F$ in $\Q[c]$. If $r=1$, this statement says simply
that $B_1=1$ divides $F$; the products and root lists are empty.

Putting $t=c+r/2$ in \eqref{eq:general-reflection} yields
\begin{equation}\label{eq:F-parity}
F(-c)=(-1)^dF(c).
\end{equation}
Replacing $j$ by $r-j$ in \eqref{eq:B-product} gives
\begin{equation}\label{eq:B-parity}
B_r(-c)=(-1)^{r-1}B_r(c).
\end{equation}
Set $G=F/B_r$. Substitution in \eqref{eq:F-parity} and
cancellation of $B_r(c)$ using \eqref{eq:B-parity} give
\begin{equation}\label{eq:R-parity}
G(-c)=(-1)^{d-r+1}G(c)
=(-1)^{\eps}G(c).
\end{equation}
If $\eps=0$, take $p=G$. If $\eps=1$, evaluation at $c=0$
gives $G(0)=0$, so $G=cp$ for a polynomial $p$.
Equation~\eqref{eq:R-parity} then becomes
$(-c)p(-c)=-cp(c)$, whence cancellation of $c$ gives $p(-c)=p(c)$.

The identity $F(r/2)=L_P(0)=1$ shows that $F$ is nonzero, so its
factorization determines a unique nonzero $p$. Since $\deg F=d$ and
$\deg B_r=r-1$, comparison of degrees in \eqref{eq:general-factor} yields
\[
\deg p=d-(r-1)-\eps=\ell\geq0.
\]
In particular, $r-1\leq d$.

\subsection{Integer-valued polynomials and finite differences}
We use the following finite-difference formula for integer-valued
polynomials.

\begin{lemma}[Newton expansion]\label{lem:newton}
Let $f\in\Q[t]$ have degree at most $d$ and satisfy $f(n)\in\Z$ for all
$n\in\NN$. Define $\Delta f(t)=f(t+1)-f(t)$. Then
\begin{equation}\label{eq:newton}
f(t)=\sum_{k=0}^d b_k\binom tk,\qquad
b_k=\Delta^k f(0)=\sum_{j=0}^k(-1)^{k-j}\binom kj f(j)\in\Z.
\end{equation}
Consequently $d!f(t)\in\Z[t]$.
\end{lemma}
\begin{proof}
The polynomials $\binom tk$, $0\leq k\leq d$, have distinct degrees
$k$ and nonzero leading coefficients $1/k!$. Successive subtraction of
leading terms shows that they form a basis for the polynomials of degree
at most $d$. Write $f=\sum_{k=0}^d b_k\binom tk$ in this basis.
Pascal's identity, interpreted as a polynomial identity, gives
\[
\Delta\binom tk=\binom t{k-1}\quad(k\geq1),\qquad
\Delta\binom t0=0.
\]
After $j$ differences, the summands with $k<j$ vanish, the summand with
$k=j$ becomes $b_j$, and the terms with $k>j$ vanish at $t=0$, because
$\binom0{k-j}=0$. Thus $b_j=\Delta^j f(0)$.
The identity
\[
\Delta^k f(t)=\sum_{j=0}^k(-1)^{k-j}\binom kj f(t+j)
\]
follows by induction on $k$. For the induction step, subtract its value
at $t$ from its value at $t+1$. The coefficient of $f(t+j)$ for
$1\leq j\leq k$ becomes
$(-1)^{k+1-j}\bigl(\binom k{j-1}+\binom kj\bigr)
=(-1)^{k+1-j}\binom{k+1}j$; the two endpoint terms have the same
form. Setting $t=0$ proves the second formula in \eqref{eq:newton}.
Each summand is integral.
Finally,
\[
d!\binom tk=\frac{d!}{k!}\prod_{j=0}^{k-1}(t-j)\in\Z[t]
\]
for every $k\leq d$. Multiply \eqref{eq:newton} by $d!$ and sum to obtain
the last assertion.
\end{proof}

\begin{lemma}[Exact division by a monic polynomial]\label{lem:monic}
Suppose $g\in\Z[x]$ is monic and $f\in\Z[x]$. If $f=gq$ for some
$q\in\Q[x]$, then $q\in\Z[x]$.
\end{lemma}
\begin{proof}
Divide $f$ by $g$. If the current dividend has
leading term $b x^n$ and $\deg g=m\leq n$, subtract
$b x^{n-m}g$. The coefficient $b$ is integral, and the new dividend is
again an integer polynomial of smaller degree. Repeating finitely often
gives $f=gq_0+r_0$, where $q_0,r_0\in\Z[x]$ and either $r_0=0$ or
$\deg r_0<m$. Subtracting this identity from $f=gq$ gives
$g(q-q_0)=r_0$. A nonzero left side has degree at least $m$, whereas a
nonzero right side has degree below $m$. Hence both vanish and $q=q_0$.
\end{proof}

\subsection{Proof of the denominator bound}
Apply Lemma~\ref{lem:newton} to $f=L_P$, and write
$U(t)=d!L_P(t)\in\Z[t]$. If $r$ is even, $r/2\in\Z$, so
\[
d!F(c)=U(c-r/2)\in\Z[c].
\]
The polynomial
\[
D(c)=(r-1)!c^\eps B_r(c)
=c^\eps\prod_{j=1}^{r-1}(c-r/2+j)
\]
is monic in $\Z[c]$. By \eqref{eq:general-factor},
\[
d!F(c)=D(c)\left(\frac{d!}{(r-1)!}p(c)\right).
\]
Lemma~\ref{lem:monic} shows that the polynomial in parentheses has
integer coefficients. This proves the first case of
\eqref{eq:general-den}.

If $r$ is odd, then $r/2$ is half-integral. Set $c=x/2$ and write
$U(t)=\sum_{k=0}^d u_k t^k$, with $u_k\in\Z$. Then
\begin{equation}\label{eq:half-shift-integral}
2^d d!F(x/2)=2^d U((x-r)/2)
=\sum_{k=0}^d u_k2^{d-k}(x-r)^k\in\Z[x].
\end{equation}
Set $m=r-1+\eps$, so $\ell=d-m$, and define the monic integer polynomial
\[
D^{\ast}(x)=x^\eps\prod_{j=1}^{r-1}(x-r+2j).
\]
Direct substitution in \eqref{eq:general-factor} gives
\begin{align}
F(x/2)&=\frac{D^{\ast}(x)}{2^m(r-1)!}\,p(x/2),\notag\\
2^d d!F(x/2)&=D^{\ast}(x)
\left(\frac{2^\ell d!}{(r-1)!}\,p(x/2)\right).
\label{eq:half-shift-division}
\end{align}
By \eqref{eq:half-shift-integral} and Lemma~\ref{lem:monic}, the
polynomial in parentheses belongs to $\Z[x]$. Substitute $x=2c$ to
deduce
\[
\frac{2^\ell d!}{(r-1)!}\,p(c)\in\Z[c].
\]
Because $r-1\leq d$, the displayed multiplier is a positive integer.
If a reduced coefficient is $b/q$ and multiplying it by an integer $K$
gives an integer, then $q\mid K$ since $\gcd(b,q)=1$. Taking the least
common multiple over the coefficients proves the claimed divisibility
in \eqref{eq:general-den}. A factor of $d!$ has no prime greater than
$d$, and the additional factor is a power of $2$. This completes the
proof of Theorem~\ref{thm:general}.

\section{The quasi-symmetric polytope}\label{sec:qs-model}

Let $V_a\subset\R^{a\times a}$ be the linear subspace defined by
$u_{ij}=u_{ji}$ for $i+j\ne a+1$, and set
\begin{equation}\label{eq:qs-polytope}
\begin{split}
\Lambda_a&=V_a\cap\Z^{a\times a},\\
P_a&=\{u\in V_a:0\leq u_{ij}\leq1,\
u_{ij}\geq u_{i+1,j},\ u_{ij}\geq u_{i,j+1}\}.
\end{split}
\end{equation}
Every displayed adjacent-index inequality is imposed only when its two
cells belong to $[a]^2$.

\begin{lemma}\label{lem:qs-dimension}
The space $V_a$ has dimension $N_a$, the group $\Lambda_a$ is a full
lattice in $V_a$, and $P_a$ has dimension $N_a$.
\end{lemma}
\begin{proof}
A fully symmetric $a\times a$ matrix has $a$ diagonal coordinates and
$\binom a2$ off-diagonal coordinates, hence $a(a+1)/2$ independent
entries. On the anti-diagonal the nonfixed transpose pairs number
$\lfloor a/2\rfloor$. Each such pair is disjoint from the other pairs;
removing its one equality replaces one variable by two and increases the
dimension by one. This gives $N_a$.
Choose one representative of every pair whose equality is retained,
and keep both coordinates of each freed pair. The selected entries
determine the full matrix by copying entries. Arbitrary integer values of
these entries give an integer matrix, and every integer matrix in $V_a$
arises this way. Thus this coordinate map identifies $\Lambda_a$ with
$\Z^{N_a}$.

For the dimension of $P_a$, define
\begin{equation}\label{eq:qs-witness}
u^*_{ij}=\frac{2a+1-i-j}{2a}.
\end{equation}
This matrix is symmetric. Its smallest entry is $1/(2a)$, its largest
is $(2a-1)/(2a)$, and each adjacent difference is $1/(2a)$. Thus every
inequality in \eqref{eq:qs-polytope} is strict at $u^*$. For example,
perturbing each coordinate within $V_a$ by less than $1/(4a)$ in absolute
value keeps all adjacent differences positive and all bounds strict.
Lemma~\ref{lem:strict} therefore proves $\aff(P_a)=V_a$.
\end{proof}

\begin{lemma}\label{lem:qs-integrality}
Every vertex of $P_a$ belongs to $\Lambda_a$.
\end{lemma}
\begin{proof}
For $u\in P_a$ and $0<t\leq1$, define the threshold array
\[
E_t(u)_{ij}=\begin{cases}1,&u_{ij}\geq t,\\0,&u_{ij}<t.\end{cases}
\]
If $u_{ij}\geq u_{i+1,j}$, then $u_{i+1,j}\geq t$ implies
$u_{ij}\geq t$; hence $E_t(u)_{ij}\geq E_t(u)_{i+1,j}$. The identical
implication for the column inequality gives
$E_t(u)_{ij}\geq E_t(u)_{i,j+1}$. Equal entries have equal thresholds,
so all defining equalities of $V_a$ are preserved. Thus
$E_t(u)\in P_a\cap\{0,1\}^{a\times a}$.

For $u\ne0$, list its distinct positive entries as
$0=t_0<t_1<\cdots<t_k\leq1$. Set $E^{(j)}:=E_{t_j}(u)$
for $1\leq j\leq k$. Then
\begin{equation}\label{eq:threshold-ordinary}
u=\sum_{j=1}^k(t_j-t_{j-1})E^{(j)}+(1-t_k)0.
\end{equation}
Indeed, a coordinate equal to $t_b$ appears in exactly the first $b$
threshold arrays; the coefficient sum there is $t_b$. A zero coordinate
remains zero. The weights are nonnegative and sum to one. The zero
matrix has the trivial decomposition. Consequently $P_a$ is the convex
hull of its finite set of feasible $0/1$ matrices.

Combine repeated points in such a convex combination.
If an extreme point were not one of the
points used, its combination would involve at least two distinct points
with positive weights. Separating one term from the rest would express
it as a nontrivial convex combination of two different points of $P_a$,
contradicting the definition of an extreme point. Thus all vertices
are feasible $0/1$ matrices and are integral.
\end{proof}

The definitions now give, including $H=0$,
\begin{equation}\label{eq:qs-count}
A_a(H)=\#(HP_a\cap\Lambda_a).
\end{equation}
Ehrhart polynomiality therefore applies with dimension $N_a$.

\subsection{Relation with ordinary order polytopes}
We identify $P_a$ with an order polytope in the sense of
Stanley \cite[Definition 1.1]{Stanley}.
A partially ordered set is a set with a reflexive, antisymmetric,
transitive relation. Start with the product order on $[a]^2$, and
identify $(i,j)$ with $(j,i)$ only when $i+j\ne a+1$.
Define directed edges between the resulting classes using the adjacent
horizontal and vertical relations, and take reachability as the order.
The integer $i+j-2$ is unchanged by every identification and increases
by exactly one on every edge. Hence no directed cycle is possible, and
reachability is antisymmetric. This constructs a poset with $N_a$ elements.

In the representative coordinates, $P_a$ consists of order-reversing
maps from this poset into $[0,1]$. Replacing every value $u$ by $1-u$
identifies this convention with the usual order-preserving order
polytope by an affine lattice isomorphism. In our order-reversing
convention, the vertices are the characteristic functions of order
ideals, as follows from \cite[Corollary 1.3]{Stanley}. The identification
also gives the counting interpretation in \eqref{eq:qs-count};
compare \cite[Theorem 4.1]{Stanley}.

\section{The staircase and complement argument for quasi-symmetry}
\label{sec:qs-interior}

Recall the integer staircase
\begin{equation}\label{eq:delta}
\delta_{ij}=2a+1-i-j.
\end{equation}
Its adjacent differences equal one, and
\begin{equation}\label{eq:delta-identities}
\delta_{ij}=\delta_{ji},\qquad
\delta_{ij}+\delta_{\rho(i,j)}=2a,\qquad
\delta_{ij}+\delta_{\tau(i,j)}=2a.
\end{equation}
For instance, $\delta_{\rho(i,j)}=i+j-1$, which proves the second
identity; the third follows by the symmetry of $\delta$.

\begin{lemma}\label{lem:qs-interior}
For positive integers $H<2a$, $HP_a$ has no relative interior lattice
point. For every integer $H\geq2a$,
\begin{equation}\label{eq:qs-interior-translation}
\relint(HP_a)\cap\Lambda_a
=\delta+\bigl((H-2a)P_a\cap\Lambda_a\bigr).
\end{equation}
\end{lemma}
\begin{proof}
By Lemma~\ref{lem:strict} and \eqref{eq:qs-witness}, a lattice matrix
$\pi$ is in $\relint(HP_a)$ precisely when it satisfies the equalities
of $V_a$ and
\begin{equation}\label{eq:strict-array}
1\leq\pi_{ij}\leq H-1,\qquad
\pi_{ij}-\pi_{i+1,j}\geq1,\quad
\pi_{ij}-\pi_{i,j+1}\geq1.
\end{equation}
A path of south and east unit steps from $(i,j)$ to $(a,a)$ has
$(a-i)+(a-j)$ steps. Summing its inequalities and using
$\pi_{aa}\geq1$ gives
\begin{equation}\label{eq:lower-delta}
\pi_{ij}\geq1+(a-i)+(a-j)=\delta_{ij}.
\end{equation}

For the upper bound, consider the complementary array
\begin{equation}\label{eq:complement}
\bar\pi_{ij}=H-\pi_{\rho(i,j)}.
\end{equation}
This sends integer entries strictly between $0$ and $H$ to entries in
the same range. Moreover
\begin{align*}
\bar\pi_{ij}-\bar\pi_{i+1,j}
&=\pi_{a-i,a+1-j}-\pi_{a+1-i,a+1-j}\geq1,\\
\bar\pi_{ij}-\bar\pi_{i,j+1}
&=\pi_{a+1-i,a-j}-\pi_{a+1-i,a+1-j}\geq1.
\end{align*}
If $i+j\ne a+1$, the reflected cell also lies off the anti-diagonal.
The defining equality for $\pi$ thus implies
$\bar\pi_{ij}=\bar\pi_{ji}$. Since
$H-\bar\pi_{\rho(i,j)}=\pi_{ij}$, complementation is an involution on
the relative interior lattice points.
Apply \eqref{eq:lower-delta} to $\bar\pi$ at the cell
$\rho(i,j)$:
\[
H-\pi_{ij}\geq\delta_{\rho(i,j)}=2a-\delta_{ij}.
\]
Hence
\begin{equation}\label{eq:both-bounds}
\delta_{ij}\leq\pi_{ij}\leq H-2a+\delta_{ij}.
\end{equation}
For $H<2a$ these bounds contradict each other, proving emptiness.

If $H\geq2a$, set $\pi'=\pi-\delta$. Equation~\eqref{eq:both-bounds}
gives $0\leq\pi'_{ij}\leq H-2a$. For either existing adjacent cell
$q$ immediately below or to the right of $p$, one has
\[
\pi'_p-\pi'_q=(\pi_p-\pi_q)-(\delta_p-\delta_q)
=(\pi_p-\pi_q)-1\geq0.
\]
The equality $\delta_{ij}=\delta_{ji}$ preserves the required transpose
equalities. Thus $\pi'\in(H-2a)P_a\cap\Lambda_a$.
Conversely, if $\pi'$ is in this latter set, then
$\pi=\pi'+\delta$ has adjacent differences at least one. Since
$1\leq\delta_{ij}\leq2a-1$, its entries lie between $1$ and $H-1$.
It belongs to $V_a$, so it is an interior lattice point. Addition and
subtraction of $\delta$ are inverse operations, proving the set equality.
At $H=2a$, the unique interior point is $\delta$, corresponding to
the zero matrix in $0P_a$.
\end{proof}

\begin{proof}[Proof of Theorem~\ref{thm:qs}]
Lemmas~\ref{lem:qs-dimension}, \ref{lem:qs-integrality}, and
\ref{lem:qs-interior} verify the hypotheses of
Theorem~\ref{thm:general} with $d=N_a$, $r=2a$, and $\omega=\delta$.
In particular the reciprocity identity is
\[
A_a(-H)=(-1)^{N_a}A_a(H-2a).
\]
For $a=2b$, $N_a=2b(b+1)$ is even and
$N_a-2a+1$ is odd. Thus $\eps=1$ and
\[
\ell=2b(b+1)-(4b-1)-1=2b(b-1).
\]
For $a=2b+1$, $N_a=2b^2+4b+1$ is odd and $N_a-2a+1$ is even.
Thus $\eps=0$ and
\[
\ell=(2b^2+4b+1)-(4b+1)=2b^2.
\]
These are respectively
$\lfloor(2b-1)^2/2\rfloor$ and $\lfloor(2b)^2/2\rfloor$.
The binomial factor in Theorem~\ref{thm:general} becomes
$\binom{c+a-1}{2a-1}$. Its even-$r$ denominator bound is exactly
\eqref{eq:main-qs-degree-den}. This also includes $a=1$.
\end{proof}

\section{The two polytopes with signed identifications}
\label{sec:signed-models}

For the second-kind families, center every entry at the half-height:
\begin{equation}\label{eq:centering}
z_{ij}=\pi_{ij}-h.
\end{equation}
Then the bounds become $-h\leq z_{ij}\leq h$, the order inequalities are
unchanged, and \eqref{eq:qt} becomes the homogeneous equation
\begin{equation}\label{eq:signed-equalities}
z_{\tau(i,j)}=-z_{ij}\qquad(i\ne j).
\end{equation}
Allowing real entries subject to these relations, define
\begin{align}
W_a&=\{z\in\R^{a\times a}:z_{\tau(i,j)}=-z_{ij}
\text{ whenever }i\ne j\},\notag\\
W_a^{\mathrm s}&=\{z\in W_a:z_{ij}=z_{ji}\text{ for all }i,j\},
\label{eq:signed-spaces}\\
\Gamma_a&=W_a\cap\Z^{a\times a},\qquad
\Gamma_a^{\mathrm s}=W_a^{\mathrm s}\cap\Z^{a\times a}.\notag
\end{align}
The two unit polytopes are
\begin{equation}\label{eq:signed-polytopes}
\begin{split}
R_a&=\{z\in W_a:-1\leq z_{ij}\leq1,\
z_{ij}\geq z_{i+1,j},\ z_{ij}\geq z_{i,j+1}\},\\
R_a^{\mathrm s}&=R_a\cap W_a^{\mathrm s}.
\end{split}
\end{equation}
The term signed refers to the coordinate relations
\eqref{eq:signed-equalities}.

\begin{lemma}[Independent coordinates]\label{lem:signed-coordinates}
The spaces $W_a$ and $W_a^{\mathrm s}$ have dimensions $d_a$ and
$e_a$, respectively, and their coordinate lattices are isomorphic
to $\Z^{d_a}$ and $\Z^{e_a}$.
\end{lemma}
\begin{proof}
The main diagonal is invariant under $\tau$, and
\eqref{eq:signed-equalities} is not imposed there. Hence all $a$
diagonal entries are independent real variables.
Among the $a^2-a$ off-diagonal cells, the cells fixed by $\tau$ are
exactly those on the anti-diagonal. There are
\[
f_a=\begin{cases}a,&a\text{ even},\\a-1,&a\text{ odd}\end{cases}
\]
such cells off the main diagonal. At a fixed cell the equality is
$z_{ij}=-z_{ij}$, forcing $z_{ij}=0$. Every remaining off-diagonal
$\tau$-orbit has two cells with opposite entries, and contributes one
independent variable. These orbits are disjoint and exhaust all
equations, so
\begin{equation}\label{eq:dimension-signed-count}
\dim W_a=a+\frac{a^2-a-f_a}{2}
=\begin{cases}a^2/2,&a\text{ even},\\(a^2+1)/2,&a\text{ odd}.
\end{cases}
\end{equation}
This is $d_a$.

For $W_a^{\mathrm s}$, first identify the off-diagonal cells by
transposition. They become the $\binom a2$ unordered pairs
$\{i,j\}$, $i\ne j$. Complementation acts on these pairs by
\[
\{i,j\}\longmapsto\{a+1-i,a+1-j\}.
\]
An unordered pair is fixed precisely when $i+j=a+1$. There are
$\lfloor a/2\rfloor$ such pairs, whose entries are forced to zero.
All other pairs occur in two-element orbits with opposite entries.
The diagonal remains unrestricted by the complement equations.
It follows that
\begin{equation}\label{eq:dimension-symmetric-count}
\dim W_a^{\mathrm s}
=a+\frac{\binom a2-\lfloor a/2\rfloor}{2}.
\end{equation}
For $a=2b$ this is $b(b+1)$; for $a=2b+1$ it is $(b+1)^2$.
These values equal $\lfloor(a+1)^2/4\rfloor=e_a$.

In the first model, select every diagonal entry and one cell from each
nonfixed off-diagonal $\tau$-orbit. In the symmetric model, select
every diagonal entry and one unordered pair from each nonfixed orbit
of unordered pairs. The remaining entries are recovered by copying,
changing signs, or setting a fixed coordinate to zero. Arbitrary
integer assignments to the selected coordinates give exactly the
integer matrices in the corresponding space. This proves the
lattice isomorphisms.
\end{proof}

\begin{lemma}[Full dimension and interior inequalities]
\label{lem:signed-dimension}
The dimensions of $R_a$ and $R_a^{\mathrm s}$ are $d_a$ and $e_a$.
In either defining space, the relative interior is given by strict
bounds $-1<z_{ij}<1$ and strict adjacent order inequalities.
\end{lemma}
\begin{proof}
Let
\begin{equation}\label{eq:eta}
\eta_{ij}=a+1-i-j,\qquad z^*_{ij}=\frac{\eta_{ij}}{a}.
\end{equation}
One has $\eta_{ji}=\eta_{ij}$ and
\[
\eta_{\tau(i,j)}
=a+1-(a+1-j)-(a+1-i)=-\eta_{ij}.
\]
Thus $z^*$ belongs to both linear spaces. Its entries range from
$-(a-1)/a$ to $(a-1)/a$, and every adjacent difference equals $1/a$.
All inequalities in \eqref{eq:signed-polytopes} are therefore strict
at $z^*$. Lemma~\ref{lem:strict}, applied separately in the two
spaces, proves the assertions.
\end{proof}

\begin{lemma}[Odd threshold decomposition]\label{lem:signed-integrality}
Both $R_a$ and $R_a^{\mathrm s}$ are lattice polytopes in their stated
lattices.
\end{lemma}
\begin{proof}
For a matrix $z$ in either polytope and $0<t\leq1$, define
\begin{equation}\label{eq:T}
T_t(z)_{ij}=\begin{cases}
1,&z_{ij}\geq t,\\
0,&-t<z_{ij}<t,\\
-1,&z_{ij}\leq-t.
\end{cases}
\end{equation}
The threshold rule is nondecreasing in the entry and changes sign when
the entry changes sign, including at the endpoints $\pm t$.
Thus $z_p\geq z_q$ implies $T_t(z)_p\geq T_t(z)_q$, and
$z_{\tau(p)}=-z_p$ implies $T_t(z)_{\tau(p)}=-T_t(z)_p$.
Equal entries also have equal thresholds. Consequently $T_t(z)$
belongs to the same polytope as $z$ and has entries in $\{-1,0,1\}$.

For a nonzero feasible matrix $z$, list its distinct positive absolute
entry values as $0=t_0<t_1<\cdots<t_k\leq1$.
Set $T^{(j)}:=T_{t_j}(z)$ for $1\leq j\leq k$. We have the finite
convex decomposition
\begin{equation}\label{eq:threshold-signed}
z=\sum_{j=1}^k(t_j-t_{j-1})T^{(j)}+(1-t_k)0.
\end{equation}
At a coordinate $z_p=t_b>0$, the first $b$ threshold matrices have
entry $1$ and all later matrices have entry zero, so the right side
is $t_b$. At $z_p=-t_b$, the first $b$ entries are $-1$, giving
$-t_b$. A zero coordinate has zero in every summand. This proves
the identity coordinate by coordinate. Its coefficients are
nonnegative with sum one. The case $z=0$ is immediate.
Each polytope is consequently the convex hull of its finite set of
feasible $\{-1,0,1\}$ matrices. The extreme-point argument in the
proof of Lemma~\ref{lem:qs-integrality} shows that all vertices are
in that finite set. Lemma~\ref{lem:signed-coordinates} identifies
these integer matrices with lattice points in the independent
coordinates, proving the claim.
\end{proof}

For an integer $h\geq0$, addition of $h\one$ is a bijection from
$hR_a\cap\Gamma_a$ to $\QT_a(h)$. It changes bounds $[-h,h]$ to $[0,2h]$, preserves
adjacent differences, and changes the complementary sum zero to
$2h$. Its inverse is \eqref{eq:centering}. It preserves symmetry,
so the analogous statement holds for the symmetric model. Thus
\begin{equation}\label{eq:signed-counts}
Q_a(h)=\#(hR_a\cap\Gamma_a),\qquad
S_a(h)=\#(hR_a^{\mathrm s}\cap\Gamma_a^{\mathrm s}).
\end{equation}
The shift is integral since $h$ is an integer. Theorem~\ref{thm:ehrhart}
now proves polynomiality of degrees $d_a$ and $e_a$.

\section{Interior translations in the signed models}
\label{sec:signed-interior}

The centered staircase satisfies
\begin{equation}\label{eq:delta-eta}
\delta=\eta+a\one.
\end{equation}
Subtraction of $\delta$ from a plane partition corresponds to subtraction
of $\eta$ after centering at the respective half-heights.

\begin{lemma}\label{lem:signed-interior}
For $1\leq h<a$, neither $hR_a$ nor $hR_a^{\mathrm s}$ has a relative
interior lattice point. For every integer $h\geq a$,
\begin{align}
\relint(hR_a)\cap\Gamma_a
&=\eta+\bigl((h-a)R_a\cap\Gamma_a\bigr),\label{eq:interior-R}\\
\relint(hR_a^{\mathrm s})\cap\Gamma_a^{\mathrm s}
&=\eta+\bigl((h-a)R_a^{\mathrm s}\cap\Gamma_a^{\mathrm s}\bigr).
\label{eq:interior-Rs}
\end{align}
\end{lemma}
\begin{proof}
Let $z\in\relint(hR_a)\cap\Gamma_a$ and put
$\pi=z+h\one$. Lemma~\ref{lem:signed-dimension} implies that the
integer entries of $\pi$ satisfy
\begin{equation}\label{eq:signed-strict-pi}
1\leq\pi_{ij}\leq2h-1,\qquad
\pi_{ij}-\pi_{i+1,j}\geq1,\quad
\pi_{ij}-\pi_{i,j+1}\geq1,
\end{equation}
as well as \eqref{eq:qt}. The path from $(i,j)$ to $(a,a)$ used in
\eqref{eq:lower-delta} depends only on these inequalities, and gives
$\pi_{ij}\geq\delta_{ij}$.

In the present box, the complement is
$\bar\pi_{ij}=2h-\pi_{\rho(i,j)}$.
As in \eqref{eq:complement}, it preserves the strict bounds and
adjacent inequalities. It also preserves the complementary equations:
$\rho$ and
$\tau$ commute, since $\tau=\rho s$ and $\rho s=s\rho$, and
$\rho$ preserves the off-diagonal set. Hence for $i\ne j$,
\begin{align*}
\bar\pi_{ij}+\bar\pi_{\tau(i,j)}
&=4h-\bigl(\pi_{\rho(i,j)}+\pi_{\rho\tau(i,j)}\bigr)\\
&=4h-\bigl(\pi_{\rho(i,j)}+\pi_{\tau\rho(i,j)}\bigr)=2h.
\end{align*}
Applying the lower bound to $\bar\pi$ at $\rho(i,j)$ yields
\begin{equation}\label{eq:signed-bounds}
\delta_{ij}\leq\pi_{ij}\leq2(h-a)+\delta_{ij}.
\end{equation}
If $h<a$, the upper endpoint is smaller than the lower endpoint,
so no such $\pi$, and hence no such $z$, exists.

For $h\geq a$, set $\pi'=\pi-\delta$. The bounds in
\eqref{eq:signed-bounds} give $0\leq\pi'_{ij}\leq2(h-a)$.
Subtracting the adjacent staircase difference one from each
inequality in \eqref{eq:signed-strict-pi} proves weak decrease of
$\pi'$ along both rows and columns. For each off-diagonal cell,
\begin{align*}
\pi'_{ij}+\pi'_{\tau(i,j)}
&=\pi_{ij}+\pi_{\tau(i,j)}-
  \bigl(\delta_{ij}+\delta_{\tau(i,j)}\bigr)\\
&=2h-2a=2(h-a).
\end{align*}
Thus $\pi'\in\QT_a(h-a)$.

Conversely, let $\pi'\in\QT_a(h-a)$ and set
$\pi=\pi'+\delta$. Every adjacent difference of $\pi$ is at
least one. The bounds on $\pi'$ and $1\leq\delta_{ij}\leq2a-1$
give $1\leq\pi_{ij}\leq2h-1$. Adding the two staircase entries at
a $\tau$-pair changes the complementary sum $2(h-a)$ to $2h$.
Therefore $\pi-h\one$ is in $\relint(hR_a)\cap\Gamma_a$.
These two constructions are inverse in uncentered coordinates.

The smaller partition $\pi'$ has half-height $h-a$. Its centered
matrix is therefore $z'=\pi'-(h-a)\one$.
Using \eqref{eq:delta-eta}, we obtain
\begin{align*}
z'&=\pi'-(h-a)\one\\
  &=\pi-\delta-(h-a)\one\\
  &=(z+h\one)-(\eta+a\one)-(h-a)\one\\
  &=z-\eta.
\end{align*}
Hence the inverse map is $z=z'+\eta$, and these maps give a bijection
between $\relint(hR_a)\cap\Gamma_a$ and $(h-a)R_a\cap\Gamma_a$,
proving \eqref{eq:interior-R}. The vector $\eta$ belongs to
$\Gamma_a$ by \eqref{eq:eta}.

For the symmetric model, the same steps apply with the extra
equalities $\pi_{ij}=\pi_{ji}$. Complementation preserves these
equalities because $s$ commutes with $\rho$. Addition and subtraction
of $\delta$ preserve them since $\delta_{ij}=\delta_{ji}$.
The centered translation preserves them since $\eta$ is symmetric.
Every forward and inverse map therefore restricts to the symmetric
sets, proving \eqref{eq:interior-Rs}, with
$\eta\in\Gamma_a^{\mathrm s}$. At $h=a$ the unique interior point
in either centered polytope is $\eta$.
\end{proof}

\section{Proofs of the two second-kind factorization theorems}
\label{sec:signed-factor}

\begin{proof}[Proof of Theorem~\ref{thm:qt}]
By \eqref{eq:signed-counts}, $Q_a$ is the Ehrhart polynomial of $R_a$.
Lemmas~\ref{lem:signed-dimension}, \ref{lem:signed-integrality},
and \ref{lem:signed-interior} allow us to apply
Theorem~\ref{thm:general} with $d=d_a$, $r=a$, and $\omega=\eta$.
Hence
\[
Q_a(-h)=(-1)^{d_a}Q_a(h-a),\qquad
B_a(c)=\binom{c+a/2-1}{a-1}.
\]
It remains to compute the parity of $d_a-a+1$.
If $a=2b$, then $d_a=2b^2$ and $d_a-a+1=2b^2-2b+1$ is odd.
If $a=2b+1$, then $d_a=2b^2+2b+1$ and $d_a-a+1=2b^2+1$ is odd.
In both cases the $\eps$ in Theorem~\ref{thm:general} equals one.
This proves the additional factor $c$ and the evenness of $\qtp$.
Its degree is
\[
\ell_a=d_a-a=
\begin{cases}2b(b-1),&a=2b,\\2b^2,&a=2b+1,\end{cases}
\]
equivalently $a(a-2)/2$ for even $a$ and $(a-1)^2/2$ for odd $a$.
The two cases of \eqref{eq:general-den}, with $r=a$, give precisely
\eqref{eq:main-qt-den}. Nonzeroness and uniqueness follow from the
general theorem.
\end{proof}

\begin{proof}[Proof of Theorem~\ref{thm:st}]
Apply Theorem~\ref{thm:general} to $R_a^{\mathrm s}$ in
$\Gamma_a^{\mathrm s}$, with $d=e_a$, $r=a$, and $\omega=\eta$.
To determine the parity, first take $a=2b$.
Then $e_a=b(b+1)$ is even, so $e_a-a+1$ is odd.
Next take $a=2b+1$. Then
\[
e_a-a+1=(b+1)^2-(2b+1)+1=b^2+1.
\]
This is even exactly when $b$ is odd, which is equivalent to
$a\equiv3\pmod4$. Therefore the $\eps$ from the general theorem
is exactly $\beta_a$, proving \eqref{eq:main-st}. Subtracting the
degrees of its factors gives
$m_a=e_a-(a-1)-\beta_a$. Explicitly,
\begin{equation}\label{eq:symmetric-degrees-expanded}
m_a=\begin{cases}
b(b-1),&a=2b,\\
4b^2,&a=4b+1,\\
4b^2+4b+2,&a=4b+3.
\end{cases}
\end{equation}
These are nonnegative even integers in the stated range $a\geq2$.
The denominator assertions are the even-$r$ and odd-$r$ cases of
\eqref{eq:general-den}.
\end{proof}

\begin{remark}[Half-integral centers]
If $a$ is odd, the center $a/2$ is a half-integer. The expressions
$Q_a(c-a/2)$ and $S_a(c-a/2)$ are nonetheless polynomials in $c$.
Their counting interpretation is on $c\in a/2+\NN$; their
factorizations and parity identities hold in $\Q[c]$.
\end{remark}

\section{Explicit small cases and a size-five formula}
\label{sec:examples}

We derive small-size formulas from the defining inequalities and the
degree bounds in Theorems~\ref{thm:qs}--\ref{thm:st}.

\subsection{The first quasi-symmetric enumerators}
For $a=1$, the sole entry has $H+1$ choices, so
\[
A_1(H)=H+1,\qquad p_1(c)=1.
\]
For $a=2$, the only nonfixed transpose pair lies on the anti-diagonal,
so all plane partitions in the box are quasi-symmetric. Write
\[
\pi=\begin{pmatrix}u&v\\w&t\end{pmatrix},\qquad
0\leq t\leq v,w\leq u\leq H.
\]
For fixed $u,t$, the entries $v,w$ are independent and each has
$u-t+1$ choices. Put $j=u-t$. For a fixed $j\in\{0,\ldots,H\}$,
the pairs $(u,t)$ are $(j,0),(j+1,1),\ldots,(H,H-j)$, numbering
$H-j+1$. Hence, with $k=j+1$,
\begin{align*}
A_2(H)&=\sum_{j=0}^H(H-j+1)(j+1)^2\\
&=(H+2)\sum_{k=1}^{H+1}k^2-\sum_{k=1}^{H+1}k^3\\
&=(H+1)(H+2)^2\left(\frac{2H+3}{6}-\frac{H+1}{4}\right)\\
&=\frac{(H+1)(H+2)^2(H+3)}{12}.
\end{align*}
Here the identities
$\sum_{k=1}^n k^2=n(n+1)(2n+1)/6$ and
$\sum_{k=1}^n k^3=n^2(n+1)^2/4$ follow by induction: their
right sides vanish at $n=0$ and their successive differences are
$n^2$ and $n^3$, respectively. Substituting $H=c-2$ gives
\begin{equation}\label{eq:qs2}
A_2(c-2)=\frac{(c-1)c^2(c+1)}{12}
=c\binom{c+1}{3}\frac12.
\end{equation}

For $a=3$, Theorem~\ref{thm:qs} says that
$p_3(c)=u c^2+v$. The value $A_3(0)=1$ is immediate.
At height one the matrices have the form
\[
\begin{pmatrix}b_1&b_2&b_3\\b_2&b_4&b_5\\b_6&b_5&b_7\end{pmatrix},
\qquad b_1\geq b_2\geq b_3,b_4,b_6\geq b_5\geq b_7,
\]
where all entries are zero or one, and $b_3,b_4,b_6$ have no mutual
order constraints. The cases $b_1=0$, $b_1=1,b_2=0$,
$b_1=b_2=1,b_5=0$, and $b_1=b_2=b_5=1$ contribute respectively
$1,1,8,2$, proving $A_3(1)=12$. Evaluate
\eqref{eq:main-qs} at $c=3,4$ to obtain
\[
9u+v=1,\qquad 6(16u+v)=12.
\]
Subtracting gives $7u=1$, and then $v=-2/7$. Thus
\begin{equation}\label{eq:qs3}
A_3(c-3)=\binom{c+2}{5}\frac{c^2-2}{7}.
\end{equation}
In this instance the residual denominator is $7$, whereas the
general bound is $7!/5!=42$.

\subsection{Second-kind enumerators for sizes two and three}
At size two, the complement equations fix the two off-diagonal
entries at $h$. The diagonal entries are independently chosen from
$[h,2h]\cap\Z$ and $[0,h]\cap\Z$. Every such matrix is symmetric,
so
\begin{equation}\label{eq:q2s2}
Q_2(h)=S_2(h)=(h+1)^2,
\qquad q_2=s_2=1.
\end{equation}

For size three, a centered integer matrix in $hR_3$ has the form
\[
z=\begin{pmatrix}
x&p&0\\q&y&-p\\0&-q&t
\end{pmatrix}.
\]
The order and bound inequalities are equivalent to
\[
0\leq p,q\leq h,\quad
M\leq x\leq h,\quad -m\leq y\leq m,\quad
-h\leq t\leq-M,
\]
where $M=\max(p,q)$ and $m=\min(p,q)$. Necessity follows by
reading each row and column; conversely these inequalities imply
each adjacent comparison in the displayed matrix. Consequently
\begin{align}
Q_3(h)&=\sum_{p=0}^h\sum_{q=0}^h
 (h-\max(p,q)+1)^2(2\min(p,q)+1),\label{eq:q3-sum}\\
S_3(h)&=\sum_{p=0}^h(h-p+1)^2(2p+1).
\label{eq:s3-sum}
\end{align}
The second identity follows from symmetry, which is exactly the
condition $p=q$. Thus
$Q_3(0)=S_3(0)=1$, $Q_3(1)=4+1+1+3=9$, and $S_3(1)=4+3=7$.

Both residual degrees are two. Writing the residuals as $u c^2+v$,
Theorem~\ref{thm:qt} at $c=3/2,5/2$ gives
\[
\frac94u+v=\frac23,\qquad
\frac{25}{4}u+v=\frac65.
\]
Subtracting gives $4u=8/15$, so $u=2/15$ and $v=11/30$.
For the symmetric class, Theorem~\ref{thm:st} gives instead
\[
\frac94u+v=1,\qquad
\frac{25}{4}u+v=\frac73,
\]
so $u=1/3$ and $v=1/4$. Therefore
\begin{align}
Q_3(c-3/2)&=c\binom{c+1/2}{2}\frac{4c^2+11}{30},
\label{eq:q3-final}\\
S_3(c-3/2)&=\binom{c+1/2}{2}\frac{4c^2+3}{12}.
\label{eq:s3-final}
\end{align}
By Theorems~\ref{thm:qt} and \ref{thm:st}, each residual lies in the
space spanned by $1$ and $c^2$. The two distinct values of $c^2$
therefore determine it uniquely.

\subsection{The symmetric class at size five}

\begin{proposition}\label{prop:s5}
For the symmetric second-kind class under \eqref{eq:qt},
\begin{equation}\label{eq:s5-corrected}
S_5(c-5/2)=c\binom{c+3/2}{4}
\frac{16c^4+64c^2-17}{2520}.
\end{equation}
Equivalently,
\begin{equation}\label{eq:s5-refactored}
S_5(c-5/2)=c\binom{c+3/2}{4}\binom{c+1/2}{2}
\frac{4c^2+17}{315}.
\end{equation}
\end{proposition}
\begin{proof}
Every centered symmetric matrix satisfying the complementary
equalities has a unique representation
\begin{equation}\label{eq:s5-matrix}
z=\begin{pmatrix}
x_1&p&q&r&0\\
p&x_2&s&0&-r\\
q&s&x_3&-s&-q\\
r&0&-s&x_4&-p\\
0&-r&-q&-p&x_5
\end{pmatrix}.
\end{equation}
Because the matrix is symmetric, its column inequalities are exactly
the transposes of its row inequalities. Reading the five rows shows
that feasibility in $hR_5^{\mathrm s}$ is equivalent to
\begin{equation}\label{eq:s5-constraints}
\begin{gathered}
0\leq q\leq p\leq h,\qquad 0\leq r\leq q,\qquad0\leq s\leq q,\\
p\leq x_1\leq h,\qquad s\leq x_2\leq p,\qquad
-s\leq x_3\leq s,\\
-p\leq x_4\leq-s,\qquad -h\leq x_5\leq-p.
\end{gathered}
\end{equation}
For example, the middle row gives $q\geq s\geq x_3\geq-s\geq-q$,
which is equivalent to $q\geq s\geq0$ and $-s\leq x_3\leq s$.
The first and last rows give $p\geq q\geq r\geq0$ and the displayed
bounds on $x_1,x_5$. The second and fourth rows give the bounds
on $x_2,x_4$. This proves the equivalence with
\eqref{eq:s5-constraints}.

For fixed $p,q,s$, the variable $r$ has $q+1$ choices. The five
diagonal entries have independently
\[
(h-p+1),\quad(p-s+1),\quad(2s+1),\quad(p-s+1),\quad(h-p+1)
\]
choices. It follows that
\begin{equation}\label{eq:s5-sum}
S_5(h)=\sum_{p=0}^h(h-p+1)^2
\sum_{q=0}^p(q+1)\sum_{s=0}^q(p-s+1)^2(2s+1).
\end{equation}
For $p=0,1,2$ the two inner sums are respectively
\begin{align*}
M_0&=1,\\
M_1&=4+2(4+3)=18,\\
M_2&=9+2(9+12)+3(9+12+5)=129.
\end{align*}
Thus $S_5(0)=1$, $S_5(1)=4+18=22$, and
$S_5(2)=9+4\cdot18+129=210$.

Since $e_5=9$ and $\beta_5=1$, Theorem~\ref{thm:st} proves that
$s_5(c)=u c^4+v c^2+w$. Evaluate the factorization at
$c=5/2,7/2,9/2$, corresponding to $h=0,1,2$. Its known factor
$c\binom{c+3/2}{4}$ has values $5/2,35/2,135/2$, respectively.
Therefore
\begin{equation}\label{eq:s5-system}
\begin{aligned}
\frac{625}{16}u+\frac{25}{4}v+w&=\frac25,\\
\frac{2401}{16}u+\frac{49}{4}v+w&=\frac{44}{35},\\
\frac{6561}{16}u+\frac{81}{4}v+w&=\frac{28}{9}.
\end{aligned}
\end{equation}
Subtract successive equations to obtain
\[
111u+6v=\frac67,\qquad260u+8v=\frac{584}{315}.
\]
Three times the second equation minus four times the first gives
$336u=32/15$, so $u=2/315$. The first difference equation then
gives $v=8/315$, and the first equation of
\eqref{eq:s5-system} gives $w=-17/2520$. This proves
\eqref{eq:s5-corrected}. Finally,
\[
16c^4+64c^2-17=(4c^2-1)(4c^2+17),\qquad
\binom{c+1/2}{2}=\frac{4c^2-1}{8},
\]
which proves \eqref{eq:s5-refactored}.
\end{proof}

\begin{remark}[Comparison with the printed size-five expression]
In the size-five formula in
\cite[Appendix A.3, p.~21]{SchreierAigner}, the last factor appears
as $(4c^4+17)/315$ following the same two binomial factors as in
\eqref{eq:s5-refactored}. Such an expression has degree $11$, whereas
the degree of $S_5$ is $9$. Under the normalization \eqref{eq:qt},
Proposition~\ref{prop:s5} gives the final factor $(4c^2+17)/315$.
\end{remark}

\section{Irreducibility and further questions}
\label{sec:limits}

An irreducible polynomial in $\Q[c]$ is, by convention, a nonzero
nonunit polynomial that cannot be written as a product of two
nonunits. The units of $\Q[c]$ are the nonzero constants. Thus
$p_1=1$, $p_2=1/2$, and
$q_2=s_2=1$ are units, not
irreducible elements.

The size-three examples yield irreducible residuals. A quadratic over
$\Q$ is reducible exactly when it has a rational root: a proper
factorization must have two linear factors, and a rational root
conversely gives a linear factor by division. The roots of $c^2-2$
are irrational. Indeed, if a reduced fraction $u/v$ satisfied
$(u/v)^2=2$, then $u^2=2v^2$ would force $u$ even and then $v$
even, a contradiction. This proves irreducibility in
\eqref{eq:qs3}. The quadratics $4c^2+11$ and $4c^2+3$ have no
real root, so in particular no rational root; the two size-three
residuals in \eqref{eq:q3-final} and \eqref{eq:s3-final} are also
irreducible over $\Q$.

In the symmetric family, Proposition~\ref{prop:s5} gives the reducible
size-five residual
\[
s_5(c)=\frac{(4c^2-1)(4c^2+17)}{2520}.
\]
For this family, the irreducibility question in
\cite[Section~4.1]{SchreierAigner} is restricted to even sizes.

\subsection{Finite-range computation}

\begin{proposition}[Computer-assisted verification]\label{prop:finite-range}
For every integer $3\leq a\leq24$, the residual polynomial $p_a(c)$
is irreducible in $\Q[c]$.
\end{proposition}

We use the following elementary criterion.
\begin{lemma}\label{lem:factor-degrees}
Let $f\in\Z[c]$ be primitive of degree $D\geq1$, and let $\mathcal L$
be a nonempty finite set of primes not dividing its leading coefficient.
For each $\ell\in\mathcal L$, let $\mathcal D_\ell$ be the set of subset
sums of the degrees of the irreducible factors of $f\bmod\ell$,
with multiplicities included. If
\[
\{1,\ldots,\lfloor D/2\rfloor\}
\cap\bigcap_{\ell\in\mathcal L}\mathcal D_\ell=\varnothing,
\]
then $f$ is irreducible over $\Q$.
\end{lemma}
\begin{proof}
By Gauss's lemma, a proper factorization would give $f=uv$ with
$u,v\in\Z[c]$ and $1\leq\deg u\leq\lfloor D/2\rfloor$.
Neither factor loses degree modulo any $\ell\in\mathcal L$.
Unique factorization over $\mathbb F_\ell$ therefore implies
$\deg u\in\mathcal D_\ell$ for every $\ell$, a contradiction.
\end{proof}

\begin{proof}[Computational verification]
Let $\mathcal P_a$ be the poset in Section~\ref{sec:qs-model}.
A labeling of height at most $H$ corresponds to $H$ weakly decreasing order ideals,
given by the sets of elements with label at least $1,\ldots,H$;
the label of an element is recovered by counting the sets containing it.
For the computation, partition $\mathcal P_a$ into blocks whose elements
have identical strict predecessor and successor sets.
Write $m_i$ for the size of block $i$ and $\mathbf m=(m_i)_i$.
For each feasible occupancy vector $\mathbf k$, fix a representative
ideal containing $k_i$ elements of block $i$, and let $w_H(\mathbf k)$
count the $H$-term weakly decreasing sequences of ideals contained in it.
Permutations within each block preserve the order, so this representative
contains exactly $\prod_i\binom{k_i}{j_i}$ subideals of any feasible
occupancy $\mathbf j\leq\mathbf k$. Thus
\begin{equation}\label{eq:weighted-ideal}
\begin{gathered}
w_0(\mathbf k)=1,\qquad
w_{H+1}(\mathbf k)=
\sum_{\substack{\mathbf j\leq\mathbf k\\\mathbf j\text{ feasible}}}
\left(\prod_i\binom{k_i}{j_i}\right)w_H(\mathbf j),\\
A_a(H)=w_H(\mathbf m).
\end{gathered}
\end{equation}

Put $2k=\deg p_a$. Theorem~\ref{thm:qs} gives
\begin{equation}\label{eq:modular-interpolation}
\frac{p_a(c)}{(2a-1)!}
=\frac{A_a(c-a)}{c^{\alpha_a}\prod_{j=1-a}^{a-1}(c-j)}
=\sum_{i=0}^k b_i c^{2i},\qquad N_a!b_i\in\Z.
\end{equation}
For a prime $\ell>\max\{N_a,2(a+k+2)\}$, compute
$A_a(H)\bmod\ell$ for $0\leq H\leq k+2$ using
\eqref{eq:weighted-ideal}. At $c=a+H$, all displayed denominator
factors are nonzero modulo $\ell$, and the points $(a+H)^2$ for
$0\leq H\leq k$ are distinct modulo $\ell$: their pairwise differences
factor into a nonzero difference and a positive sum, both smaller than
$\ell$ in absolute value. These $k+1$ values determine the $b_i$
modulo $\ell$; the remaining two heights give independent checks.
Reductions with a vanishing leading coefficient are discarded.
Since $N_a!$ is invertible modulo $\ell$, every retained reduction
differs by a nonzero scalar from the reduction of the primitive integer
polynomial associated with $p_a$.

The resulting polynomials are factored in $\mathbb F_\ell[c]$, not
in a variable representing $c^2$, using FLINT 3.6.0 through
\texttt{python-flint} 0.9.0 \cite{FLINT,PythonFlint}.
Each factor is checked by Rabin's irreducibility criterion \cite{Rabin},
and the factors, with multiplicities and leading scalar, are multiplied
back to recover the input. For $a=6,7,10$, reductions are instead taken
from primitive integer polynomials reconstructed from exact counts
using the established degree bound; the same degree-preservation and
factor checks apply.
For every $3\leq a\leq24$, the retained factorizations give an empty
intersection in Lemma~\ref{lem:factor-degrees}, proving the assertion.
All arithmetic is exact. The compressed counts also agree with exact
integer counts at all $423$ tested height values for $1\leq a\leq17$.
\end{proof}

\subsection{Open arithmetic questions}
Theorems~\ref{thm:qs}--\ref{thm:st} leave the general irreducibility
questions for $p_a$ and $q_a$ at $a\geq3$, and for $s_a$ at even
$a\geq4$. Proposition~\ref{prop:finite-range} settles only the
quasi-symmetric family in its stated finite range. It gives no
verification for $a\geq25$ or for the other two families beyond the
individual examples proved above. Determining the exact coefficient
denominators, beyond the divisibility bounds in the three theorems,
is a further arithmetic problem.

\section*{Declaration on the Use of Generative AI}
Parts of the proofs in this paper were generated using the generative AI
tools ChatGPT and Codex. The author has reviewed and verified the relevant
proofs.


\begin{thebibliography}{99}

\bibitem{BeckRobins}
M.~Beck and S.~Robins,
\emph{Computing the Continuous Discretely: Integer-Point Enumeration
in Polyhedra}, 2nd ed., Undergraduate Texts in Mathematics,
Springer, New York, 2015.
\href{https://doi.org/10.1007/978-1-4939-2969-6}
{doi:10.1007/978-1-4939-2969-6}.

\bibitem{Ehrhart}
E.~Ehrhart,
Sur les poly\`edres rationnels homoth\'etiques \`a $n$ dimensions,
\emph{C.~R. Acad. Sci. Paris} \textbf{254} (1962), 616--618.

\bibitem{FLINT}
The FLINT development team,
\emph{FLINT: Fast Library for Number Theory},
documentation, \texttt{nmod\_poly\_factor} module,
\url{https://flintlib.org/doc/nmod_poly_factor.html},
accessed September 13, 2026.

\bibitem{Macdonald}
I.~G.~Macdonald,
Polynomials associated with finite cell-complexes,
\emph{J. London Math. Soc.} (2) \textbf{4} (1971), 181--192.
\href{https://doi.org/10.1112/jlms/s2-4.1.181}
{doi:10.1112/jlms/s2-4.1.181}.

\bibitem{PythonFlint}
The python-flint development team,
\emph{python-flint}, version 0.9.0,
\url{https://github.com/flintlib/python-flint}.

\bibitem{Rabin}
M.~O.~Rabin,
Probabilistic algorithms in finite fields,
\emph{SIAM J. Comput.} \textbf{9} (1980), no.~2, 273--280.
\href{https://doi.org/10.1137/0209024}
{doi:10.1137/0209024}.

\bibitem{SchreierAigner}
F.~Schreier-Aigner,
Fully complementary higher dimensional partitions,
\emph{Ann. Comb.} \textbf{29} (2025), 1--23.
\href{https://doi.org/10.1007/s00026-024-00691-5}
{doi:10.1007/s00026-024-00691-5}.

\bibitem{Stanley}
R.~P.~Stanley,
Two poset polytopes,
\emph{Discrete Comput. Geom.} \textbf{1} (1986), 9--23.
\href{https://doi.org/10.1007/BF02187680}
{doi:10.1007/BF02187680}.

\end{thebibliography}
\end{document}